\documentclass[11pt, reqno]{amsart}

\usepackage[T1]{fontenc}
\usepackage{lmodern}

\usepackage{amssymb}
\usepackage{amsmath}
\usepackage{mathtools}

\usepackage{aliascnt}

\usepackage[letterpaper, margin=1.25in]{geometry}

\usepackage[hidelinks]{hyperref}

\usepackage[nameinlink, capitalise, noabbrev]{cleveref}

\numberwithin{equation}{section}

\theoremstyle{plain}

\newaliascnt{theorem}{equation}
\newtheorem{theorem}[theorem]{Theorem}
\aliascntresetthe{theorem}

\newtheorem{maintheorem}{Theorem}

\crefname{maintheorem}{Theorem}{Theorems}
\Crefname{maintheorem}{Theorem}{Theorems}

\newaliascnt{proposition}{equation}
\newtheorem{proposition}[proposition]{Proposition}
\aliascntresetthe{proposition}

\newaliascnt{lemma}{equation}
\newtheorem{lemma}[lemma]{Lemma}
\aliascntresetthe{lemma}

\newaliascnt{corollary}{equation}
\newtheorem{corollary}[corollary]{Corollary}
\aliascntresetthe{corollary}

\theoremstyle{definition}

\newaliascnt{definition}{equation}

\aliascntresetthe{definition}

\newaliascnt{example}{equation}

\aliascntresetthe{example}

\theoremstyle{remark}

\newaliascnt{remark}{equation}

\aliascntresetthe{remark}

\DeclareMathOperator{\vol}{Vol}

\newcommand{\R}{\mathbb R}
\newcommand{\C}{\mathbb C}
\newcommand{\dd}{\,\mathrm d}
\DeclareMathOperator{\Tr}{Tr}
\DeclareMathOperator{\diver}{div}
\DeclareMathOperator{\wind}{wind}
\DeclareMathOperator{\ord}{ord}

\DeclarePairedDelimiter{\paren}{(}{)}
\renewcommand{\b}{\paren*}

\newcommand{\pd}[3][]{\frac{\partial^{#1}#2}{\partial #3^{#1}}}

\title{Geometric rigidity from the Dirichlet-to-Neumann operator}
\author{Romain Speciel}
\address{Department of Mathematics, Stanford University, Stanford, CA 94305, USA}
\date{September 14, 2026}
\email{rspeciel@stanford.edu}

\begin{document}

\begin{abstract}
    The analytic properties of the Dirichlet-to-Neumann operator reflect the geometry of the underlying manifold. As a guiding example of this relationship, consider the Euclidean unit ball. There, the Dirichlet-to-Neumann operator has an explicitly computable spectrum, consisting of the nonnegative integers with multiplicity, and commutes with the boundary Laplacian. In this paper, we study the extent to which these properties distinguish the ball among various classes of manifolds. We show that, in every dimension, Euclidean balls are uniquely determined by their Steklov spectrum among bounded Euclidean domains with smooth and possibly disconnected boundary, answering a well-known open question. We then classify all smooth compact connected oriented surfaces with nonempty boundary whose Dirichlet-to-Neumann operator commutes with the boundary Laplacian. Finally, in dimensions three and above, we prove that, for smooth metrics conformal to the Euclidean ball, the Dirichlet-to-Neumann operator commutes with the Laplacian on the round sphere if and only if the conformal factor is radial.
\end{abstract}

\maketitle

\section{Introduction}

Let $(M,g)$ be a smooth compact connected Riemannian manifold of dimension $n\geq2$ with nonempty boundary. Define the \emph{Dirichlet-to-Neumann operator}
\[
    \Lambda_g\colon C^\infty(\partial M)\to C^\infty(\partial M)\qquad\text{by}\qquad
    f\mapsto \Lambda_g f=\pd{u_f}{\nu_g}\bigg|_{\partial M},
\]
where $\nu_g$ denotes the outward-pointing unit normal to $\partial M$ and $u_f$ solves
\[
    \begin{cases}
        \Delta_g u_f=0&\text{in }M,\\
        u_f|_{\partial M}=f&\text{on }\partial M.
    \end{cases}
\]
It is standard that $\Lambda_g$ is an elliptic classical pseudodifferential operator of order one with principal symbol $|\xi|_g$ for $(x,\xi)\in T^*\partial M\setminus\{0\}$. Its realization on $L^2(\partial M)$ is nonnegative and self-adjoint with compact resolvent \cite[Section~7.1]{LevitinMangoubiPolterovich2023}. Its eigenvalues therefore form a sequence $0=\sigma_0<\sigma_1\leq\dots\to\infty$, where each term is repeated according to the dimension of its eigenspace. This sequence is called the \emph{Steklov spectrum} of $(M,g)$ as the eigenvalue equation $\Lambda_g f=\sigma f$ is equivalent to the \emph{Steklov problem}, which asks for nonzero $u\in C^\infty(M)$ solving
\[
    \begin{cases}
        \Delta_g u=0&\text{in }M,\\
        \pd{u}{\nu_g}=\sigma u&\text{on }\partial M.
    \end{cases}
\]
Note that, throughout this paper, we use the nonnegative Laplacian. Consult \cite{ColboisGirouardGordonSher2024} for a recent survey of this problem.

Consider the Euclidean ball as a guiding example. For $r>0$, write $\mathbb{B}_r^n=\{x\in\mathbb R^n:|x|<r\}$ and let $\mathbb{B}^n=\mathbb{B}_1^n$ with $\mathbb{S}^{n-1}=\partial \mathbb{B}^n$. Recall that eigenfunctions of the Laplacian on $\mathbb{S}^{n-1}$, called \emph{spherical harmonics}, are restrictions of homogeneous harmonic polynomials on $\mathbb R^n$. A spherical harmonic $Y_\ell$ of degree $\ell$ therefore has harmonic extension $u(r,\omega)=r^\ell Y_\ell(\omega)$ to $\mathbb{B}^n$, with outward normal derivative $\ell Y_\ell$ on the boundary. From the expression for the Euclidean Laplacian in polar coordinates, we find
\[
    \Lambda_{\mathbb{B}^n}Y_\ell=\ell Y_\ell
    \qquad\textrm{and}\qquad 
    \Delta_{\mathbb{S}^{n-1}}Y_\ell=\ell(\ell+n-2)Y_\ell,
\]
from which we deduce \cite[Example~1.1]{S25}
\begin{equation}\label{eq:ball-identity}
    \Delta_{\mathbb{S}^{n-1}}=\Lambda_{\mathbb{B}^n}^2+(n-2)\Lambda_{\mathbb{B}^n}.
\end{equation}
From these observations, we record two consequences. First, the Steklov eigenvalues of $\mathbb{B}_r^n$ are $\ell/r$, with eigenspaces given by spherical harmonics of degree $\ell$ \cite[Example~1.3.2]{GirouardPolterovich2017}. Second, the boundary Laplacian on the ball may be written as a function of its Dirichlet-to-Neumann operator, and in particular these two operators commute. In this paper, we study the extent to which these two properties uniquely distinguish the Euclidean ball within various classes of manifolds.

\subsection{Spectral rigidity of Euclidean balls} We prove that, in dimensions $n\geq3$, if the Steklov spectrum of a bounded Euclidean domain $\Omega$ with smooth boundary is asymptotically close to that of a ball, then $\Omega$ must itself be a ball.

\begin{maintheorem}
\label{thm:ball-rigidity}
	Let $\Omega\subset \mathbb{R}^n$, $n\geq 3$, be a bounded domain with smooth boundary, and write $\sigma_k$ and $s_k(r)$ for the $k$th Steklov eigenvalue of $\Omega$ and $\mathbb{B}^n_r$, respectively. Suppose that, for some $r>0$,
	\[
		\sigma_k-s_k(r)=o\b{k^{-1/(n-1)}}\quad \text{as}\quad k\to \infty.
	\]
	Then $\Omega$ is a ball of radius $r$.
\end{maintheorem}

The proof of \cref{thm:ball-rigidity}, given in \cref{sec:ball-rigidity}, is strongly inspired by the work in \cite{PolterovichSher2015}. For $n\geq4$, equality of the first three Steklov heat invariants with those of a ball forces the mean curvature to be constant and the boundary to be umbilic. The boundary volume then shows that there is only one component. When $n=3$, we first deduce periodicity of the boundary geodesic flow, as in \cite[Section~5]{PolterovichSher2015}, through a novel approach using Egorov's Theorem. Each boundary component is consequently a topological sphere, and the third heat coefficient shows that the boundary is connected and umbilic.

In the planar case, the spectrum of any simply connected bounded domain with smooth boundary approaches that of the disc of the same perimeter at a rate faster than any negative power of the index $k$ \cite[Theorem~1.4]{GirouardParnovskiPolterovichSher2014}. Therefore, no analogue of \cref{thm:ball-rigidity} can hold. Nevertheless, the exact Steklov spectrum of a compact surface with smooth boundary determines the number and lengths of its boundary components \cite[Theorem~1.7]{GirouardParnovskiPolterovichSher2014}. Together with Weinstock's inequality \cite{Weinstock1954} (see also \cite[Theorem~3.1]{ColboisGirouardGordonSher2024}), this shows that the disc is uniquely determined by its Steklov spectrum among bounded planar domains with smooth boundary. Combining this with \cref{thm:ball-rigidity}, we obtain the following.

\begin{corollary}
\label{cor:ball-spectral-rigidity}
	Let $\Omega\subset \mathbb{R}^n$, $n\geq 2$, be a bounded domain with smooth boundary. Suppose its Steklov spectrum is equal to that of $\mathbb{B}_r^n$ for some $r>0$. Then $\Omega$ is a ball of radius $r$.
\end{corollary}

\noindent As discussed, the $n=2$ case follows from \cite[Corollary~1.13]{GirouardParnovskiPolterovichSher2014}. When $n=3$ and under the additional assumption that the boundary be connected, this result is due to Polterovich and Sher \cite[Theorem~1.15]{PolterovichSher2015}. \cref{cor:ball-spectral-rigidity} resolves the open question posed in \cite[discussion following Corollary~1.12]{GirouardParnovskiPolterovichSher2014}, \cite[Open~Problem~8]{GirouardPolterovich2017}, and \cite[Conjecture~1.17]{PolterovichSher2015}.

\subsection{Commutativity with the boundary Laplacian}

Returning to the setting of a general manifold $(M,g)$, write $\Delta_{\partial M}$ for the Laplacian on the boundary with the metric induced by $g$. While the identity \eqref{eq:ball-identity} holds to leading order for a general manifold since $\Lambda_g^2$ and $\Delta_{\partial M}$ have the same principal symbol, the lower-order terms may not agree. It is therefore natural to ask when the boundary Laplacian is given exactly by a function of the Dirichlet-to-Neumann map. Such a relationship implies that the operators commute, so we consider the broader condition
\[
    [\Lambda_g,\Delta_{\partial M}]=0,
\]
following the question posed in \cite[Section~2.2]{GirouardKarpukhinLevitinPolterovich2022}.

Once again, the cases $n=2$ and $n\geq3$ are different. On a surface, conformal changes that leave the boundary metric fixed preserve both the Dirichlet-to-Neumann map and the boundary Laplacian. To state our second result, we must therefore introduce the following notion of equivalence, as in \cite[Definition~2.6]{ColboisGirouardGordonSher2024}. Say two Riemannian surfaces $(M_1,g_1)$ and $(M_2,g_2)$ are \emph{$\sigma$-isometric} if there are a diffeomorphism $\Phi:M_1\to M_2$ and a smooth real function $\psi$ on $M_1$ such that $\Phi^*g_2=e^{2\psi}g_1$ and $\psi|_{\partial M_1}=0$.

\begin{maintheorem}\label{thm:surfaces}
Let $(M,g)$ be a smooth compact connected oriented Riemannian surface with nonempty boundary. Then $[\Lambda_g,\Delta_{\partial M}]=0$ if and only if $(M,g)$ is $\sigma$-isometric to a round Euclidean disc, a logarithmic oval, or a flat cylinder.
\end{maintheorem}

\noindent Here a \emph{logarithmic oval} is the image of the closed unit disc under $z\mapsto\log(z-a)$ for some $a>1$, up to translation, rotation, and positive scaling, equipped with the Euclidean metric. A \emph{flat cylinder} is $[0,h]\times(\R/\ell\mathbb Z)$ with its product metric, where $h,\ell>0$.

In \cite[Theorem~1.3]{S25}, we proved this classification for surfaces of genus zero or with at least three boundary components. We remove this restriction in \cref{sec:surfaces} by combining the Gunning--Narasimhan theorem \cite[p.~103]{GunningNarasimhan} with the argument principle and the Gauss--Bonnet formula. \cref{thm:surfaces} resolves the open question posed in \cite[Open~Problem~1.6]{S25} and answers the classification question in \cite[Section~2.2, p.~203]{GirouardKarpukhinLevitinPolterovich2022} for oriented surfaces.

In dimensions $n\geq3$, a bounded domain $\Omega\subset\R^n$ with smooth boundary has commuting Dirichlet-to-Neumann map and boundary Laplacian if and only if $\Omega$ is a ball. This was proved under the additional assumption that the boundary be connected in \cite[Theorem~1.3]{GirouardKarpukhinLevitinPolterovich2022}. We removed that assumption in \cite[Theorem~1.5]{S25} and obtained a corresponding stability estimate for domains in $\R^3$ diffeomorphic to a ball in \cite[Theorem~1.6]{Speciel2026Stability}.

Finally, we consider conformal changes of the Euclidean metric on the ball. Write $g_{\mathrm{std}}$ for the Euclidean metric on $\overline{\mathbb{B}^n}$, and denote by $\Delta_{\mathbb{S}^{n-1}}$ the Laplacian on the unit sphere with the round metric, independently of the conformal factor.

\begin{maintheorem}\label{thm:conformal}
Fix $n\geq3$ and let $g_\varphi=e^{2\varphi}g_{\mathrm{std}}$ on $\overline{\mathbb{B}^n}$, where $\varphi$ is smooth and real. The following are equivalent:
\begin{enumerate}
\item $[\Lambda_{g_\varphi},\Delta_{\mathbb{S}^{n-1}}]=0$;
\item $\varphi$ is radial;
\item $\Delta_{\mathbb{S}^{n-1}}=F(\Lambda_{g_\varphi})$ for a function $F$ on the spectrum of $\Lambda_{g_\varphi}$.
\end{enumerate}
\end{maintheorem}

\noindent \cref{thm:conformal} resolves the open question posed in \cite[Open~Problem~3.22]{Speciel2026Stability} and extends our previous result on infinitesimal conformal perturbations \cite[Theorem~1.17]{Speciel2026Stability}. The equivalence between conditions (2) and (3) in \cref{thm:potential} was previously proved in \cite[Proposition~1.13]{Speciel2026Stability}.

We establish \cref{thm:conformal} in \cref{sec:radiality}. We first apply the standard conversion to a Schr\"odinger equation \cite[Section~2]{SU87} and examine the boundary symbols to reduce to the analogous result for the Schr\"odinger operator $\Delta+q$. We then prove that result using Green's identity and the complex geometrical optics solutions from \cite[Corollary~2.5]{SU87}.

The spectral and commutativity perspectives described above come together in the following consequence of \cref{thm:conformal} and the warped-product uniqueness theorem of Daud\'e, Kamran, and Nicoleau \cite[Theorem~1.3]{DKN2021}.

\begin{corollary}
Fix $n\geq3$ and, for $j=1,2$, let $g_j=e^{2\varphi_j}g_{\mathrm{std}}$ on $\overline{\mathbb{B}^n}$, where $\varphi_j$ is smooth and real. If
\[
\begin{gathered}
  [\Lambda_{g_j},\Delta_{\mathbb{S}^{n-1}}]=0
\end{gathered}
\]
and the Steklov spectra of $(\overline{\mathbb{B}^n}, g_1)$ and $(\overline{\mathbb{B}^n}, g_2)$ agree, then $\varphi_1=\varphi_2$. In particular, such a metric whose Steklov spectrum agrees with that of the Euclidean unit ball is the Euclidean metric.
\end{corollary}

\vspace{1em}

\noindent\textbf{Acknowledgements.} The author is grateful for the comments and advice provided by Rafe Mazzeo, Iosif Polterovich, and David Sher. This work was carried out in collaboration with several AI models that contributed significantly through mathematical exploration, technical discussions, and editorial assistance. All AI-assisted proofs and computations were independently checked and understood by the author, who devoted substantial effort to simplifying the arguments and polishing their exposition. The author bears responsibility for this manuscript's content and correctness.

\section{Spectral rigidity of Euclidean balls}
\label{sec:ball-rigidity}

In this section, we prove \cref{thm:ball-rigidity}. Let $\Omega\subset\mathbb R^n$, $n\geq3$, be a bounded domain with smooth boundary $\Sigma=\partial\Omega$, and let $\Lambda$ be its Dirichlet-to-Neumann operator. Note we suppress the metric from the notation when it is evident from context.

\subsection{Heat invariants and curvature}

We first study the trace of the Steklov heat kernel to obtain constraints on the curvature of $\Sigma$, as in \cite{PolterovichSher2015}. This trace has the asymptotic expansion
\[
  \Tr(e^{-t\Lambda})
  =a_0t^{1-n}+a_1t^{2-n}+a_2t^{3-n}+o(t^{3-n})
\]
as $t\to0$. The coefficients $a_0,a_1,a_2$ are the first three \emph{Steklov heat invariants} \cite[Equation~(1.3)]{PolterovichSher2015}.

\begin{lemma}
\label{lem:heat-trace-comparison}
Under the assumptions of \cref{thm:ball-rigidity},
\[
  \Tr(e^{-t\Lambda})-\Tr(e^{-t\Lambda_{\mathbb B_r^n}})
  =o(t^{3-n})\qquad\text{as }t\to0.
\]
\end{lemma}
\begin{proof}
The Weyl law \cite[Equation~(1.2)]{PolterovichSher2015} for the two Steklov spectra gives a constant $c>0$ such that
\[
  \min\{\sigma_k,s_k(r)\}\geq c k^{1/(n-1)}
  \qquad\text{for all }k\geq1.
\]
Given $\varepsilon>0$, the decay assumption allows us to choose $k_0\geq1$ such that $|\sigma_k-s_k(r)|\leq\varepsilon k^{-1/(n-1)}$ for every $k\geq k_0$. By the mean value theorem,
\[
\begin{aligned}
  |e^{-t\sigma_k}-e^{-ts_k(r)}|
  &\leq t|\sigma_k-s_k(r)|e^{-t\min\{\sigma_k,s_k(r)\}}\\
  &\leq\varepsilon t k^{-1/(n-1)}e^{-ctk^{1/(n-1)}}
\end{aligned}
\]
for $k\geq k_0$. The mean value theorem also bounds the sum of the absolute values of the finitely many terms with $k<k_0$ by $C_\varepsilon t$. Summing over the spectrum therefore gives
\begin{align*}
  \left|\sum_{k=0}^\infty\b{e^{-t\sigma_k}-e^{-ts_k(r)}}\right|
  &\leq C_\varepsilon t+\varepsilon t
    \sum_{k=k_0}^\infty k^{-1/(n-1)}e^{-ctk^{1/(n-1)}}\\
  &\leq C_\varepsilon t+\varepsilon t
    \int_0^\infty x^{-1/(n-1)}e^{-ctx^{1/(n-1)}}\dd x,
\end{align*}
where the last inequality follows because the integrand is decreasing. We make the substitution $u=ct x^{1/(n-1)}$ to rewrite this integral as
\[
  \int_0^\infty x^{-1/(n-1)}e^{-ctx^{1/(n-1)}}\dd x
  =(n-1)(ct)^{2-n}\int_0^\infty u^{n-3}e^{-u}\dd u.
\]
Since $n\geq3$, the integral on the right-hand side converges. We conclude that
\[
  \left|\Tr(e^{-t\Lambda})-\Tr(e^{-t\Lambda_{\mathbb B_r^n}})\right|
  \leq C_\varepsilon t+C\varepsilon t^{3-n},
\]
where $C$ is independent of $\varepsilon$. Dividing by $t^{3-n}$, the first term tends to zero as $t\to0$, while the second becomes $C\varepsilon$. Since $\varepsilon$ is arbitrary, the difference of the traces is $o(t^{3-n})$.
\end{proof}

We now use the heat invariants to study the curvature of $\Sigma$. To state the relevant curvature formulas, let $\kappa_1,\dots,\kappa_{n-1}$ be the principal curvatures at some point $x\in\Sigma$. Define the \emph{mean curvature} and \emph{second-order mean curvature} by
\[
  H_1(x)=\frac{1}{n-1}\sum_j\kappa_j
  \quad\text{and}\quad
  H_2(x)=\frac{1}{(n-1)(n-2)}\sum_{j\neq k}\kappa_j\kappa_k,
\]
respectively. In particular, $H_1=1/r$ and $H_2=1/r^2$ on $\partial\mathbb B_r^n$. Introduce the dimensional constants
\[
\begin{gathered}
  C_n=\frac{\vol(\mathbb S^{n-2})\Gamma(n-1)}{(2\pi)^{n-1}},\\
  \alpha_n=\frac{n-2}{3}\b{(n-3)(3n-4)+4},\quad\text{and}\quad
  \beta_n=\frac{4(n-2)}{3(n+1)}\b{(n-4)(n+1)+3},
\end{gathered}
\]
where $\vol$ denotes Riemannian volume. The first three heat invariants are given by \cite[Theorem~1.5 and Corollary~1.10]{PolterovichSher2015}
\begin{equation}
\label{eq:heat-invariants}
\begin{gathered}
  a_0=C_n\vol(\Sigma),\qquad
  a_1=\frac{(n-2)C_n}{2}\int_\Sigma H_1\dd S,\quad\text{and}\\
  a_2=\frac{C_n}{8(n-2)}
    \int_\Sigma\b{\alpha_nH_1^2+\beta_n(H_1^2-H_2)}\dd S.
\end{gathered}
\end{equation}
\noindent We also recall the following classical observation.

\begin{lemma}
\label{lem:umbilic-hypersurfaces}
We have $H_1^2-H_2\geq0$. If equality holds identically, then each connected component of $\Sigma$ is a round sphere.
\end{lemma}
\begin{proof}
Compute
\[
  (n-1)(n-2)(H_1^2-H_2)
  =\sum_j\kappa_j^2-(n-1)H_1^2
  =\sum_j(\kappa_j-H_1)^2\geq0.
\]
The right-hand side is the squared norm of the trace-free second fundamental form. If it vanishes identically, each component of $\Sigma$ is a closed umbilic hypersurface in $\mathbb R^n$ and hence a round sphere \cite[Exercise~6(c) of Chapter~8]{doCarmo1992}.
\end{proof}

\begin{proof}[Proof of \cref{thm:ball-rigidity} for $n\geq4$]
By \cref{lem:heat-trace-comparison}, the first three heat invariants of $\Omega$ agree with those of $\mathbb B_r^n$. Note this is the only step that uses the rate of convergence in \cref{thm:ball-rigidity}. With $V=\vol(\Sigma)$, we rewrite \eqref{eq:heat-invariants} in our setting as
\begin{equation}
\label{eq:ball-heat-identities}
\begin{gathered}
  V=\vol(\partial\mathbb B_r^n),
  \qquad \int_\Sigma H_1\dd S=\frac Vr,\\
  \int_\Sigma\b{\alpha_nH_1^2+\beta_n(H_1^2-H_2)}\dd S
  =\alpha_n\frac V{r^2}.
\end{gathered}
\end{equation}
Combining these identities, we find
\[
  \int_\Sigma\b{\alpha_n\b{H_1-\frac1r}^2
    +\beta_n(H_1^2-H_2)}\dd S=0.
\]
Both coefficients $\alpha_n$ and $\beta_n$ are strictly positive when $n\geq4$. It follows that $H_1=1/r$ and $H_1^2-H_2=0$, so each component of $\Sigma$ is a sphere of radius $r$ by \cref{lem:umbilic-hypersurfaces}. Applying \eqref{eq:ball-heat-identities} shows that there is only one component. Since $\Omega$ is bounded and connected, it is the ball enclosed by this sphere.
\end{proof}

\subsection{Geodesic dynamics on the boundary}

We now consider the remaining $n=3$ case. We first show that if the Steklov spectrum of $\Omega$ tends to that of a ball, then the geodesic flow on its boundary must be periodic.

\begin{proposition}[cf. Proposition 5.2 in \cite{PolterovichSher2015} and Theorem A in \cite{Zelditch1996}]
\label{prop:periodic-geodesic-flow}
Let $\Omega\subset\mathbb R^n$, $n\geq3$, be a bounded domain with smooth boundary $\Sigma=\partial\Omega$. Suppose that, for some $r>0$,
\[
  \sigma_k-s_k(r)\to0\quad\text{as }k\to\infty.
\]
Then the geodesic flow $\Phi_t$ on the unit cotangent bundle $S^*\Sigma$ satisfies $\Phi_{2\pi r}=\operatorname{Id}$.
\end{proposition}
\begin{proof}
Set $T_0=2\pi r$ and define $K=e^{iT_0\Lambda}-\operatorname{Id}$, a bounded operator on $L^2(\Sigma)$ which admits the same eigenbasis as $\Lambda$. Since the Steklov eigenvalues of the ball $\mathbb B_r^n$ are contained in $\frac1r\mathbb Z$, we have $e^{iT_0s_k(r)}=1$ for every $k$. Taking $\phi_k$ to be an eigenfunction of $\Lambda$ corresponding to $\sigma_k$, we find
\[
  K\phi_k=\b{e^{iT_0(\sigma_k-s_k(r))}-1}\phi_k.
\]
The eigenvalues of $K$ therefore tend to zero, so $K$ is compact.

Write $\Psi^m(\Sigma)$ for the classical pseudodifferential operators of order $m$ on $\Sigma$. For $A\in\Psi^0(\Sigma)$, define
\[
  A(t)=e^{it\Lambda}Ae^{-it\Lambda}
\]
and note
\[
  A(T_0)-A=KA+AK^*+KAK^*,
\]
hence $A(T_0)-A$ is also compact. By Egorov's Theorem \cite[Theorem~7.6]{Hintz2025}, $A(t)\in\Psi^0(\Sigma)$ and its principal symbol is obtained by transporting that of $A$ along the Hamiltonian flow $\Phi_t$ of $|\xi|$. But on $S^*\Sigma$, the Hamiltonian vector fields corresponding to $|\xi|$ and $\frac12|\xi|^2$ agree, so $\Phi_t$ is the geodesic flow.

Since a compact operator of order zero has vanishing principal symbol \cite[Lemma~A.1]{Speciel2026Stability}, the principal symbol of $A$ is invariant under $\Phi_{T_0}$. However, $A$ was arbitrary, and since every smooth function on $S^*\Sigma$ is the restriction of the principal symbol of some $A\in\Psi^0(\Sigma)$, it follows that $\Phi_{T_0}$ is the identity.
\end{proof}

This result implies topological restrictions on the boundary of our domain, as expressed by the following lemma, which is a consequence of \cite[Theorem~7.37]{Besse1978}. We give a direct, simple proof in our setting.

\begin{lemma}
\label{lem:periodic-surface-sphere}
Let $M$ be a connected closed orientable Riemannian surface. If its geodesic flow satisfies $\Phi_T=\operatorname{Id}$ for some $T>0$, then $M$ is diffeomorphic to a sphere.
\end{lemma}
\begin{proof}
Fix $x\in M$. For a unit vector $v\in T_xM$, let $\gamma_v:[0,T]\to M$ be the unit-speed geodesic with initial vector $v$. Each $\gamma_v$ is a loop based at $x$, and any path of unit vectors from $v$ to $-v$ gives a based homotopy between $\gamma_v$ and $\gamma_{-v}$. By periodicity and reversibility,
\[
  \gamma_{-v}(t)=\gamma_v(T-t).
\]
Consequently $[\gamma_v]=-[\gamma_v]$ in $H_1(M;\mathbb R)$, so $[\gamma_v]=0$.

Since $x$ was arbitrary, every geodesic loop of length $T$ has zero real homology. If a prime closed geodesic has least period $\tau$, then $T=m\tau$ for some positive integer $m$, hence its $m$-fold iterate has length $T$ and so the prime geodesic also has zero real homology. Therefore, every closed geodesic has zero real homology.

The homology classes of closed geodesics span $H_1(M;\mathbb R)$, since every nontrivial free homotopy class contains a length-minimizing closed geodesic. Conclude $H_1(M;\mathbb R)=0$, and the classification of closed orientable surfaces gives $M\cong\mathbb S^2$.
\end{proof}

We are now equipped to conclude the proof of our theorem.

\begin{proof}[Proof of \cref{thm:ball-rigidity} for $n=3$]
By \cref{prop:periodic-geodesic-flow,lem:periodic-surface-sphere}, every connected component of $\Sigma$ is diffeomorphic to a sphere. Write $b$ for the number of such connected components. In this dimension, $H_2$ is the Gauss curvature, so \eqref{eq:heat-invariants} and the Gauss--Bonnet formula give
\[
\begin{aligned}
  a_2&=\frac1{16\pi}\int_\Sigma\b{H_1^2+\frac{H_2}3}\dd S\\
     &=\frac{\chi(\Sigma)}6
       +\frac1{16\pi}\int_\Sigma(H_1^2-H_2)\dd S
      =\frac b3+\frac1{16\pi}\int_\Sigma(H_1^2-H_2)\dd S,
\end{aligned}
\]
where $\chi$ denotes Euler characteristic. By \cref{lem:heat-trace-comparison}, the coefficient $a_2$ must agree with that of the ball, hence
\[
  \frac{b-1}3+\frac1{16\pi}\int_\Sigma(H_1^2-H_2)\dd S=0.
\]
Both terms are nonnegative, so both must vanish. From the first, we deduce $b=1$, while the second shows that $H_1^2-H_2=0$. Conclude that $\Sigma$ is a round sphere by \cref{lem:umbilic-hypersurfaces}. Equality of $a_0$ gives $\vol(\Sigma)=4\pi r^2$, so its radius is $r$. Since $\Omega$ is bounded and connected, it is the ball enclosed by this sphere.
\end{proof}

\section{Commuting operators on surfaces}
\label{sec:surfaces}

In this section, we prove \cref{thm:surfaces}. We therefore restrict our attention to the case when $(M,g)$ is a smooth compact connected oriented surface with nonempty boundary. For a conformal change $\widetilde g=e^{2\psi}g$,
\begin{equation}\label{eq:surface-conformal-law}
  \Delta_{\widetilde g}=e^{-2\psi}\Delta_g
  \qquad\textrm{and}\qquad
  \Lambda_{\widetilde g}=e^{-\psi|_{\partial M}}\Lambda_g,
\end{equation}
hence $\sigma$-isometry preserves both boundary operators. 

Let $T$ be the positively oriented unit tangent vector along $\partial M$, and write $\Delta_{\partial M}=-T^2$. We extend the boundary operators complex linearly and suppress boundary restrictions when no confusion is possible. A holomorphic function $f$, smooth up to the boundary, satisfies
\[
  \Lambda f=-iTf
\]
by the Cauchy--Riemann equations. We first record the converse for harmonic functions.

\begin{lemma}\label{lem:boundary-cr}
A smooth complex-valued harmonic function $u$ on $M$ is holomorphic if $\partial_\nu u=-iTu$ on $\partial M$.
\end{lemma}

\begin{proof}
Let $\partial\overline u$ denote the $(1,0)$-part of $\dd\overline u$. The boundary condition gives
\[
  (\partial\overline u)(T)
  =\tfrac12(T\overline u+i\partial_\nu\overline u)=0.
\]
Since $u$ is harmonic, the form $\partial\overline u$ is closed. By Stokes' theorem,
\[
  0=\int_{\partial M}u\,\partial\overline u
   =\int_M\overline\partial u\wedge\partial\overline u.
\]
The two-form $-i\overline\partial u\wedge\partial\overline u$ is nonnegative and vanishes precisely where $\overline\partial u=0$, hence $u$ is holomorphic.
\end{proof}

We may now establish the following key proposition.

\begin{proposition}\label{prop:surface-topology}
If $[\Lambda,\Delta_{\partial M}]=0$, then $\chi(M)\geq0$. In particular, $M$ is diffeomorphic to a disc or an annulus.
\end{proposition}

\begin{proof}
For a holomorphic function $f$ smooth up to the boundary, let $Pf$ be the smooth harmonic extension of $T^2f$. Since $[\Lambda,T^2]=0$, we have
\[
  \partial_\nu Pf=\Lambda T^2f
  =T^2\Lambda f=-iT^3f=-iT(Pf)
  \qquad\text{on }\partial M.
\]
It follows from \cref{lem:boundary-cr} that $Pf$ is holomorphic.

Extend the metric smoothly to a closed oriented surface containing $M$, and remove a point outside $M$. The Gunning--Narasimhan theorem \cite[Theorem, p.~103]{GunningNarasimhan} supplies a holomorphic function $f_0$ on the resulting open Riemann surface with $\dd f_0$ nowhere zero. In particular, $f_0$ is a holomorphic immersion on a neighborhood of $M$. Define
\[
  f_1=\tfrac12P(f_0^2)-f_0\,Pf_0.
\]
This function is holomorphic and smooth up to the boundary. The product rule gives $f_1=(Tf_0)^2$ on $\partial M$, so $f_1$ has no boundary zeros.

Equip $M$ with the flat metric $g_0=|\dd f_0|^2$. The function $f_0$ is then an orientation-preserving local isometry, and $Tf_0$ is a positive multiple of the Euclidean unit tangent along the immersed boundary. With each boundary component given its induced orientation, the Gauss--Bonnet formula therefore gives
\[
  \wind_{\partial M}(Tf_0)
  =\frac{1}{2\pi}\int_{\partial M}k_{g_0}\dd s_{g_0}
  =\chi(M),
\]
where $k_{g_0}$ is the geodesic curvature and the winding number is summed over each boundary component. Denoting by $\ord_p f_1$ the vanishing order of $f_1$ at $p$, apply the argument principle to deduce
\[
  0\leq\sum_{p\in M^\circ}\ord_p f_1
  =\wind_{\partial M}(f_1)
  =2\wind_{\partial M}(Tf_0)=2\chi(M).
\]
Since $M$ is oriented, connected, and has nonempty boundary, the inequality $\chi(M)\geq0$ forces genus zero with one or two boundary components.
\end{proof}

\begin{proof}[Proof of \cref{thm:surfaces}]
Suppose $[\Lambda,\Delta_{\partial M}]=0$. By \cref{prop:surface-topology}, $M$ has genus zero. The classification in \cite[Theorem~1.3]{S25} therefore applies, giving precisely the three models in \cref{thm:surfaces}, up to $\sigma$-isometry. Conversely, the same theorem shows that each model has commuting boundary operators, and \eqref{eq:surface-conformal-law} preserves this property under $\sigma$-isometry.
\end{proof}

\section{Radiality from a boundary commutator}
\label{sec:radiality}

In this section, we prove \cref{thm:conformal}. We convert the problem to an analogous question involving a Schr\"odinger equation, then study complex geometrical optics solutions to conclude the proof.

\subsection{Conversion to a Schr\"odinger equation}

It is well known that the Laplace equation for a conformal metric can be converted to a Schr\"odinger equation \cite[Section~2]{SU87}. For $g_\varphi=e^{2\varphi}g_{\mathrm{std}}$ as in \cref{thm:conformal}, set
\[
  s=e^{(n-2)\varphi/2},\qquad
  q=-\frac{\Delta s}{s}.
\]
The equation $\Delta_{g_\varphi}u=0$ is equivalent to $\diver(s^2\nabla u)=0$, and
\[
  s(\Delta+q)(su)=-\diver(s^2\nabla u).
\]
Multiplication by $s$ therefore converts functions harmonic with respect to $g_\varphi$ into solutions of $(\Delta+q)v=0$. Note also that zero is not a Dirichlet eigenvalue of $\Delta+q$ when $q$ arises from a conformal factor as above.

Writing $\Lambda_q$ for the Dirichlet-to-Neumann operator associated with $\Delta+q$, we will prove the following analogue of \cref{thm:conformal}.

\begin{theorem}\label{thm:potential}
Let $n\geq3$ and $q\in C^\infty(\overline{\mathbb{B}^n};\R)$, and suppose zero is not a Dirichlet eigenvalue of $\Delta+q$. The following are equivalent:
\begin{enumerate}
\item $[\Lambda_q,\Delta_{\mathbb{S}^{n-1}}]=0$;
\item $q$ is radial;
\item $\Lambda_q=F(\Delta_{\mathbb{S}^{n-1}})$ for a function $F$ on the spectrum of $\Delta_{\mathbb{S}^{n-1}}$.
\end{enumerate}
\end{theorem}

\noindent Note that commutativity alone guarantees that $\Lambda_q$ preserves each space of spherical harmonics of a fixed degree, but does not a priori ensure that its action there is scalar. 

Before proving \cref{thm:potential}, we deduce \cref{thm:conformal} by examining the boundary symbols.

\begin{proof}[Proof of \cref{thm:conformal}]
Suppose $[\Lambda_{g_\varphi},\Delta_{\mathbb{S}^{n-1}}]=0$. The principal symbol of $\Lambda_{g_\varphi}$ on the boundary is $e^{-\varphi(x)}|\xi|$, where $|\xi|$ is computed in the standard metric on the sphere \cite[Section~1.2, p.~196]{GirouardKarpukhinLevitinPolterovich2022}. Its Poisson bracket with $|\xi|^2$, the principal symbol of the Laplacian, vanishes only if $\varphi$ is constant on $\mathbb{S}^{n-1}$. Subtracting a constant from $\varphi$ rescales $\Lambda_{g_\varphi}$ and preserves all three conditions in the theorem statement, so we may assume $\varphi=0$ on the boundary.

After this rescaling, $s=1$ and $\partial_\nu s=\frac{n-2}{2}\partial_\nu\varphi$ on $\mathbb{S}^{n-1}$, where $\nu$ is the Euclidean outward unit normal. Differentiating $su$ gives
\begin{equation}\label{eq:boundary-term}
  \Lambda_q f=\Lambda_{g_\varphi}f+\frac{n-2}{2}(\partial_\nu\varphi)f,
\end{equation}
and since $\Lambda_{g_\varphi}$ and $\Lambda_{\mathbb{B}^n}$ both commute with $\Delta_{\mathbb{S}^{n-1}}$, we deduce
\begin{equation}\label{eq:commutator-with-normal-derivative}
  [\partial_\nu\varphi,\Delta_{\mathbb{S}^{n-1}}]
  =\frac{2}{n-2}[\Lambda_q-\Lambda_{\mathbb{B}^n},\Delta_{\mathbb{S}^{n-1}}],
\end{equation}
where $\partial_\nu\varphi$ acts by multiplication.

By \cite[Proposition~2.15 and equations~(2.11)--(2.12)]{NSU1995},
\[
  \Lambda_q-\Lambda_{\mathbb{B}^n}\in\Psi^{-1}(\mathbb{S}^{n-1}),
\]
so the commutator on the left-hand side of \eqref{eq:commutator-with-normal-derivative} has order at most zero. Its order-one symbol $2i\langle\nabla_{\mathbb{S}^{n-1}}(\partial_\nu\varphi),\xi\rangle$ therefore vanishes, so $\partial_\nu\varphi$ is constant on the boundary. Equation~\eqref{eq:boundary-term} gives $[\Lambda_q,\Delta_{\mathbb{S}^{n-1}}]=0$, and \cref{thm:potential} makes $q$ radial. Since $s$ and each of its rotations solve $(\Delta+q)u=0$ with boundary value $1$, uniqueness implies that $s$, and hence $\varphi$, is radial.

Conversely, suppose $\varphi$ is radial. The Dirichlet-to-Neumann operator acts on spherical harmonics of degree $\ell$ by a scalar $\lambda_\ell$, and these scalars are strictly increasing with $\ell$ by \cite[Lemma~2.3]{DKN2021}. Thus $F(\lambda_\ell)=\ell(\ell+n-2)$ defines a function satisfying $\Delta_{\mathbb{S}^{n-1}}=F(\Lambda_{g_\varphi})$. This proves~(3), which immediately implies~(1).
\end{proof}

\subsection{Radiality of the potential}

It remains to prove \cref{thm:potential}. To prove radiality, we use Green's identity to relate the boundary commutator to the angular derivatives of $q$. We then test the resulting identity against complex geometrical optics solutions, using the following form of the Sylvester--Uhlmann construction \cite[Corollary~2.5]{SU87}.

\begin{theorem}[cf. Theorem~2.3 in \cite{U14}]
\label{thm:cgo}
Let $n\geq3$ and $q\in C^\infty(\overline{\mathbb{B}^n};\R)$. There are constants $C,R>0$ such that, for every $\zeta\in\C^n$ with $\zeta\cdot\zeta=0$ and $|\zeta|\geq R$, the equation $(\Delta+q)u_\zeta=0$ has a solution $u_\zeta\in C^\infty(\overline{\mathbb{B}^n};\C)$ of the form
\[
  u_\zeta=e^{\zeta\cdot x}(1+r_\zeta),
\]
where
\[
  \|r_\zeta\|_{L^2(\mathbb{B}^n)}\leq\frac C{|\zeta|},
  \qquad
  \|\nabla r_\zeta\|_{L^2(\mathbb{B}^n)}\leq C.
\]
\end{theorem}

With these solutions, we can now prove \cref{thm:potential}.

\begin{proof}[Proof of \cref{thm:potential}]

To measure the radiality of $q$, introduce the operator
\begin{equation}\label{eq:angular-operator}
  L=-\sum_{i<j}(x_i\partial_j-x_j\partial_i)^2,
\end{equation}
which extends $\Delta_{\mathbb{S}^{n-1}}$ into the ball. Equivalently, on each sphere with its induced Euclidean metric,
\[
  (Lf)|_{\partial\mathbb{B}_r^n}
  =r^2\Delta_{\partial\mathbb{B}_r^n}(f|_{\partial\mathbb{B}_r^n}),
  \qquad 0<r\leq1.
\]
This operator commutes with $\Delta$ and with $\partial_r$ for $r>0$, and its coordinate expression in \eqref{eq:angular-operator} shows it is smooth at the origin.

Suppose $[\Lambda_q,L]=0$, and let $u,v\in C^\infty(\overline{\mathbb{B}^n};\C)$ solve $(\Delta+q)u=(\Delta+q)v=0$. Since $L$ commutes with $\Delta$,
\[
(\Delta+q)Lu=[q,L]u,
\]
so $[q,L]$ measures the failure of $L$ to map solutions of $(\Delta+q)u=0$ back to solutions.

With boundary restrictions understood, Green's identity gives
\[
\begin{aligned}
  \int_{\mathbb{B}^n}v[q,L]u\dd x
  &=\int_{\mathbb{S}^{n-1}}
       \bigl((Lu)\Lambda_qv-vL\Lambda_qu\bigr)\dd S\\
  &=\int_{\mathbb{S}^{n-1}}v[\Lambda_q,L]u\dd S=0.
\end{aligned}
\]
To express the interior commutator in terms of the angular derivatives of $q$, set
\begin{equation}\label{eq:current-field}
\begin{aligned}
  B(x)&=|x|^2\nabla q(x)-(x\cdot\nabla q(x))x\\
      &=|x|^2\operatorname{proj}_{x^\perp}\nabla q(x)\qquad(x\neq0),
\end{aligned}
\end{equation}
where $\operatorname{proj}_{x^\perp}$ is the Euclidean orthogonal projection onto the tangent space of the sphere through $x$. By definition, $B$ is smooth and tangent to every sphere centered at the origin, and $B=0$ if and only if $q$ is radial. The product rule gives
\[
  \diver B=-Lq,\qquad [q,L]=2B\cdot\nabla+\diver B.
\]
Since $B\cdot\nu=0$ on the boundary, integration by parts gives
\begin{equation}\label{eq:current}
  \int_{\mathbb{B}^n}B\cdot(v\nabla u-u\nabla v)\dd x
  =\int_{\mathbb{B}^n}v[q,L]u\dd x=0.
\end{equation}

We now test \eqref{eq:current} against the solutions in \cref{thm:cgo} to show that $B$ is a gradient. Fix $\xi\in\R^n\setminus\{0\}$ and choose orthonormal vectors $e_1,e_2\in\xi^\perp$, which is possible since $n\geq3$. For sufficiently large $\tau$, set
\[
  \zeta_\tau=\tau e_1+i\left(\xi/2+\sqrt{\tau^2-|\xi|^2/4}\,e_2\right),
  \qquad \eta_\tau=i\xi-\zeta_\tau.
\]
Then $\zeta_\tau\cdot\zeta_\tau=\eta_\tau\cdot\eta_\tau=0$, $\zeta_\tau+\eta_\tau=i\xi$, and $(\zeta_\tau-\eta_\tau)/(2\tau)\to e_1+ie_2$. Let $u_\tau,v_\tau$ be the corresponding solutions from \cref{thm:cgo}. The remainder bounds and Cauchy--Schwarz give
\[
  \frac{v_\tau\nabla u_\tau-u_\tau\nabla v_\tau}{2\tau}
  \longrightarrow e^{i\xi\cdot x}(e_1+ie_2)
  \quad\text{in }L^1(\mathbb{B}^n).
\]
Dividing \eqref{eq:current} by $2\tau$ and passing to the limit, we obtain
\[
  (e_1+ie_2)\cdot\int_{\mathbb{B}^n}e^{i\xi\cdot x}B(x)\dd x=0.
\]
Replacing $e_2$ by $-e_2$ and varying $e_1$ shows that the Fourier transform of the zero extension of $B$ is parallel to $\xi$. The zero extension of $B$ to $\mathbb{R}^n$ is therefore curl-free in the sense of distributions. By the Poincar\'e lemma, $B=\nabla f$ in the ball for a smooth function $f$. Since $B$ is tangent to the concentric spheres, $f$ is constant along each radial segment. Continuity at the origin makes $f$ constant throughout the ball. Conclude $B=0$, and \eqref{eq:current-field} then implies that $q$ is radial.

Conversely, suppose $q$ is radial. For a spherical harmonic $Y_\ell$ of degree $\ell$, the solution with boundary value $Y_\ell$ has the form $v_\ell(r)Y_\ell(\omega)$, where $v_\ell(1)=1$. The radial equation depends only on $\ell$, and the nonresonance assumption gives uniqueness. Therefore, $\Lambda_q$ acts on the whole degree space by $v_\ell'(1)$. Thus $F(\ell(\ell+n-2))=v_\ell'(1)$ defines a function satisfying $\Lambda_q=F(\Delta_{\mathbb{S}^{n-1}})$. This proves~(3), which immediately implies~(1).
\end{proof}
\bibliographystyle{amsplain}
\bibliography{bib}

\end{document}